\documentclass[preprint]{elsarticle}
\RequirePackage[top=1.5cm,bottom=1.5cm,left=2.cm,right=2.cm,includehead,includefoot]{geometry}

\usepackage{amssymb}
\usepackage{fancyhdr}
\usepackage{amsmath,amsfonts,amsthm}
\usepackage{mathrsfs}
\usepackage{amsfonts}
\usepackage{cases}
\usepackage{verbatim}
\usepackage{graphicx}      
\usepackage{epstopdf}
\usepackage{grffile}
\graphicspath{{figures/}}
\usepackage{subfig}
\usepackage{color}
\usepackage{caption}

\newtheorem{thm}{Theorem}[section]

\newtheorem{lem}[thm]{Lemma}

\theoremstyle{remark}
\newtheorem{rem}[thm]{Remark}

\DeclareMathOperator{\sech}{sech}

\begin{document}

\begin{frontmatter}
\title{Mass and energy conservative high order diagonally implicit Runge--Kutta schemes for nonlinear Schr\"odinger equation in one and two dimensions}
\tnotetext[mytitlenote]{This work was supported by the National Natural Science Foundation of China (No. 11571366, 11901577, 11971481), the
National Natural Science Foundation of Hunan (No. S2017JJQNJJ0764), the fund from Hunan Provincial Key Laboratory
of Mathematical Modeling and Analysis in Engineering (No. 2018MMAEZD004), the funds from NNW (No. NNW2018-
ZT4A08) and the Research Fund of NUDT (No. ZK17-03-27).}

\author[nudt]{Ziyuan Liu}
\ead{liuziyuan17@nudt.edu.cn}
\author[nudt]{Hong Zhang\corref{cor1}}
\ead{zhanghnudt@163.com}
\author[nudt]{Xu Qian}
\ead{qianxu@nudt.edu.cn}
\author[nudt,hpc]{Songhe Song}
\ead{shsong@nudt.edu.cn}
\cortext[cor1]{Corresponding author.}

\address[nudt]{College of Liberal Arts and Sciences, National University of Defense Technology, Changsha 410073, China}
\address[hpc]{State Key Laboratory of High Performance Computing, National University of Defense Technology, Changsha 410073, China}

\begin{abstract}
  We present and analyze a series of conservative diagonally implicit Runge--Kutta schemes for the nonlinear Schr\"odiner equation. With the application of the newly developed invariant energy
  quadratization approach, these schemes possess not only high accuracy , high order convergence (up to fifth order) and efficiency due to diagonally implicity but also mass and energy conservative properties. Both theoretical
  analysis and numerical experiments of one- and two-dimensional dynamics are carried out to verify the invariant conservative properties, convergence orders and longtime simulation stability.
\end{abstract}


\end{frontmatter}

\section{Introduction}
As the most canonical equation in quantum mechanics, the nonlinear Schr\"odinger (NLS) equation describes various physical phenomena such as the hydrodynamics \cite{zou2019conservative}, the
nonlinear optics \cite{hasegawa1989optical} and plasma physics \cite{hefter1985application}. To begin with, we consider the following one-dimensional nonlinear Schr\"odinger equation
\begin{equation}
  \label{eqn:nlse}
  \mathrm{i}u_t + \Delta u + \beta |u|^2 u = 0,
\end{equation}
subject to the initial and periodic boundary conditions
\begin{align}
    & u(x,0)=u_{0}(x), \ x \in \Omega = [a, b], \label{initial_condition} \\
    & u(a, t) = u(b, t), \ t \geq 0, \label{boundary_condition}
\end{align}
where $u$ is a complex function and $\beta$ stands for a constant real parameter. With periodic boundary condition, the NLS equation enjoys the following global mass and energy conservation law:
\begin{align}
  M(u(:, t)) &:= \int_{\Omega}|u(\mathbf{x}, t)|^2d\mathbf{x} = M(u_0), \label{eq:init-mass}\\
  E(u(:, t)) &:= \int_\Omega(-\frac{1}{2}|\nabla u(\mathbf{x}, t)|^2 + \frac{\beta}{4} |u(\mathbf{x}, t)|^4)d\mathbf{x} = E(u_0). \label{eq:init-energy}
\end{align}
The well-posedness and dynamical properties of the NLS equation were investigated in the literature \cite{sulem2007nonlinear,bao2006dynamics,makhankov1978dynamics}. Besides, there have been quantities of numerical methods developed for the NLS equation. To list a few, Akrivis et al. introduced a fully discrete Galerkin method in \cite{akrivis1991fully},
Dehghan and Mirzaei discussed the application of meshless methods on 2D NLS equation \cite{dehghan2008numerical}, Zhu et al. developed a multi-symplectic wavelet collocation method in
\cite{zhu2011multi}, Wang et al.developed a Fourth-order compact and energy conservative difference schemes \cite{wang2013fourth}, Qian et al. introduced a semi-explicit multi-symplectic splitting
scheme \cite{qian2014semi} and Y. Gong et al. proposed a conservative Fourier pseudo-spectral method \cite{gong2017conservative}. Among the proposed methods, the conservative methods have shown superiority over non-preserving ones in stability and longtime simulation.

The invariant energy quadratization (IEQ) approach is a recently developed technique that enables the methods which preserve quadratic invariant to preserve the conservation laws of arbitrary
order polynomial structures. The core idea of IEQ is to transform the original equation with high degree energy term into an equivalent coupled system by introducing a variable according to
the original energy formulation. And by this reformulation, arbitrary nonlinear terms can be treated semi-explicitly, which leads to a linear system. This approach was first proposed by Yang et al. in
\cite{yang2017linear,yang2017efficient} to preserve the evolution of a modified energy for the gradient flow, and many conservative or unconditionally energy stable methods using IEQ approach have
been studied for various equations afterwards, for instance, the Cahn--Hilliard equation\cite{yang2017numerical}, the sine-Gordon equation\cite{gong2019linearly}, the fractional Klein--Gordon-Schr\"odinger equation
\cite{wang2019linear}, the gradient flow model \cite{gong2019arbitrarily} and so on. Recently, based on IEQ approach, Jiang et al. developed a novel class of arbitrarily high order energy-preserving Runge--Kutta methods for the
Camassa--Holm eqaution \cite{jiang2019arbitrarily} and Zhang et al. introduced the high-order energy-preserving diagonally implicit Runge--Kutta methods for nonlinear Hamiltonian ODEs \cite{zhang2019diagonal}. In this article, we propose a series of mass
and energy conservative methods to discrete the reformulated NLS system by applying the Fourier pseudospectral schemes in the space direction and diagonally implicit Runge--Kutta method in the time direction for 1D and 2D NLS equations.

This article is organized as follows. In Sec. \ref{sec:2}, by utilizing the invariant energy quadratization approach, an equivalent system possessing mass and energy conservation law is introduced. After the
introducing of the spatial semi-discrete system using the Fourier pseudospectral method in Sec. \ref{sec:3}, we apply the diagonally implicit Runge--Kutta schemes which are high order and conservative
in Sec. \ref{sec:4}. Both of the semi-discrete and fully discrete systems are rigorously proven mass and energy
preserving in corresponding sections. Besides, the discretization for the two-dimensional NLS equation is also discussed in these two sections. Last but not the least, numerical experiments for both 1D and 2D cases are demonstrated in
Sec. \ref{sec:5} to verify the conservative properties, the time convergence order ranging from 2nd to 5th, the ability of capturing singularities and the stability in longtime simulation.

\section{Mass and energy preserving scheme using the IEQ approach}\label{sec:2}
Consider the domain $\Omega$ with a smooth boundary. The $L^2$ inner product and its norm are defined as $(f, g) = \int_{\Omega}f\bar{g}d\mathbf{x}$ and $\|f\|_2=\sqrt{(f, f)}$, $\forall f, g \in
L^2(\Omega)$, where $\bar g$ denotes the conjugation of $g$. By introducing an auxiliary variable $r(u) = |u|^2$, the NLS equation \eqref{eqn:nlse} can be transformed into
\begin{align}
  u_t &= \mathrm{i}u_{xx} + \mathrm{i}\beta r u, \label{eqn:fnse1}\\
  r_t & = u\bar{u}_t + \bar{u} u_t \label{eqn:fnse3}.
\end{align}
\begin{thm}
  The system \eqref{eqn:fnse1}--\eqref{eqn:fnse3} enjoys the following local energy conservation law
  \begin{align}\label{eqn:lec}
    \frac{\partial}{\partial t}(-|u_x|^2 + \frac{\beta}{2} |r|^2) + \frac{\partial}{\partial x} (u_x\bar{u}_t + \bar{u}_x u_t) = 0,
  \end{align}
\end{thm}
\begin{proof}
Multiplying Eq. \eqref{eqn:fnse1} with $-\mathrm{i}\bar{u}_t$ gives
\begin{equation}\label{eqn:add0}
\begin{aligned}
-\mathrm{i}\bar{u}_t u_t &= u_{xx} \bar{u}_t + \beta r u \bar{u}_t \\
            & = u_{xx}\bar{u}_t + u_x \bar{u}_{t x} - u_x \bar{u}_{t x} + \beta r u \bar{u}_t \\
            & = \frac{\partial}{\partial x}(u_x \bar{u}_t) - u_x \bar{u}_{t x} + \beta r u\bar{u}_t.
\end{aligned}
\end{equation}
Adding Eq. \eqref{eqn:add0} with its conjugation yields
\begin{equation}\label{eqn:lecproof}
  \begin{aligned}
    \mathrm{i}u_t \bar{u}_t -\mathrm{i}\bar{u}_t u_t & = \frac{\partial}{\partial x}(u_x \bar{u}_t + \bar{u}_x u_t) -u_x \bar{u}_{t x} - \bar{u}_x u_{tx} + \beta r(u \bar{u}_t + \bar{u} u_t) \\
    & = \frac{\partial}{\partial x}(u_x \bar{u}_t + \bar{u}_x u_t) + \frac{\partial}{\partial t}(-|u_x|^2  + \frac{\beta}{2} |r|^2).
  \end{aligned}
\end{equation}
Seeing that
\begin{equation}
  \mathrm{i}u_t \bar{u}_t -\mathrm{i}\bar{u}_t u_t = 0,
\end{equation}
the proof is complete.
\end{proof}

Furthermore, under suitable boundary conditions, such as the periodic boundary condition, the local energy conservation law (\ref{eqn:lec}) implies the global energy conservation law
\begin{align}\label{eqn:gcl}
  \frac{\mathrm{d}}{\mathrm{d} t}E(u) = \frac{\mathrm{d}}{\mathrm{d} t} \int_\Omega(-\frac{1}{2}|u_x|^2 + \frac{\beta}{4} |r|^2)dx = 0.
\end{align}

\begin{thm}
  Under periodic boundary condition, the system \eqref{eqn:fnse1}--\eqref{eqn:fnse3} admits the mass conservation law
\begin{equation}
\label{eq:mass-con}
  \frac{\mathrm{d}}{\mathrm{d} t} M(u) = \frac{\mathrm{d}}{\mathrm{d} t} \int_\Omega|u(x)|^2dx = 0.
\end{equation}
\end{thm}

\begin{proof}

Multiplying Eq. \eqref{eqn:fnse1} with $\bar{u}$ gives
\begin{equation}\label{eqn:mass0}
 \bar{u} u_t = \mathrm{i}u_{xx} \bar{u} + \mathrm{i} \beta r|u|^2.
\end{equation}
Then by integrating the above equation over $\Omega$, we obtain
\begin{equation}
\label{eq:mass1}
   \int_{\Omega}\bar u u_t dx= \mathrm{i}u_x\bar{u}|_a^b - \mathrm{i}\int_{\Omega}|u_x|^2 dx +  \mathrm{i}\beta\int_{\Omega}r|u|^2dx.
\end{equation}
Denote the real part of $u$ by $\Re(u)$. Combining the following equation
\begin{equation}
\label{eq:mass2}
  \frac{\mathrm{d}}{\mathrm{d} t} M(u) = \int_{\Omega}(\bar u u_t + u \bar{u}_t) dx = 2\Re(\int_{\Omega}\bar u u_tdx)
\end{equation}
with the real part of Eq. \eqref{eq:mass1} and the periodic boundary condition we can conclude that Eq. \eqref{eq:mass-con} holds.
\end{proof}

\section{Local energy conserving spatial semi-discretization using the Fourier pseudospectral method}\label{sec:3}
In order to discretize the system \eqref{eqn:fnse1} -- \eqref{eqn:fnse3} in the space direction, the well known Fourier pseudospectral method will be applied. The key point of the Fourier
pseudospectral method is to approximate the partial differential operators. Ref. \cite{chen2001multi} presents that the Fourier spectral differential matrix $D_k$ can
approximate the $k$th-order partial differential operator $\partial^k x$ with spectral accuracy. We present $D_1$ as follows:
\[
(D_1)_{j+1,l+1} = \left\{ \begin{array}{ll}
                            \frac{1}{2} (-1)^{j+l} \mu \cot(\mu\frac{x_j -x_l}{2}), & j\neq l, \\
                            0, & j = l, \\
                        \end{array}
                \right. for \quad j, l = 0, 1, \cdots N-1,
\]
where $\mu = \frac{2 \pi}{L}$, $L$ is the length of space interval,  the even integer $N$ is the number of the subintervals and $h=\frac{L}{N}$. It is obvious that $D_1$ is an $N\times N$ real skew-symmetric
matrix. For more details, see Ref. \cite{gong2014multi} and references therein.

By applying the Fourier pseudospectral method in the space direction of Eqns. \eqref{eqn:fnse1} -- \eqref{eqn:fnse3}, we obtain the semi-discrete system
\begin{align}
  \mathbf{u}_t &= f(\mathbf{u}, \mathbf{r}) = \mathrm{i} D_1^2 \mathbf{u} + \mathrm{i} \beta \mathbf{r} \cdot \mathbf{u}, \label{eqn:fnsesemi1}\\
  \mathbf{r}_t &= g(\mathbf{u}, \mathbf{r}) = \bar{\mathbf{u}} \cdot  \mathbf{u}_t + \mathbf{u} \cdot  \bar{\mathbf{u}}_t,\label{eqn:fnsesemi3}
\end{align}
where $\mathbf{u} = [u_0, u_1, \cdots, u_{N-1}]^T$ and $\mathbf{u} \cdot  \mathbf{v} = [u_0 v_0, u_1 v_1, \cdots, u_{N-1} v_{N-1}]^T$.

Next, we will show that this system possesses a semi-discrete modified energy conservation law.

\begin{thm}\label{thm:semi-energy}
  The semi-discrete system \eqref{eqn:fnsesemi1} -- \eqref{eqn:fnsesemi3} possesses the semi-discrete modified energy conservation law:
  \begin{equation}
    \frac{\mathrm{d}}{\mathrm{d} t}E(\mathbf{u}, \mathbf{r}) = \frac{\mathrm{d}}{\mathrm{d} t} (-\frac{1}{2}\|D_1 \mathbf{u}\|^2 + \frac{\beta}{4} \|\mathbf{r}\|^2) = 0.\label{eqn:gclsemi}
  \end{equation}
\end{thm}

\begin{proof}
We define discrete inner product and norm as:
\begin{equation}
\label{eq:discrete-norm}
  (\mathbf{u}, \mathbf{v}) := h\sum_{i=0}^{N-1} u_i \bar{v}_i, \ \|\mathbf{u}\|^2 := (\mathbf{u}, \mathbf{u}).
\end{equation}
Taking the inner product of Eq. \eqref{eqn:fnsesemi1} with $\mathrm{i} \mathbf{u}_t$ gives
\begin{equation}\label{eqn:addsemi0}
\begin{aligned}
 -\mathrm{i} ({\mathbf{u}_t}, {\mathbf{u}}_t) &= (D_1^2 \mathbf{u}, \mathbf{u}_t) + (\beta\mathbf{r} \cdot \mathbf{u}, \mathbf{u}_t) \\
    & = (D_1^2 \mathbf{u}, \mathbf{u}_t) + (\beta\mathbf{r}, \mathbf{u}_t \cdot \bar{\mathbf{u}}).
\end{aligned}
\end{equation}
Since $D_1$ is a real skew-symmetric matrix, we have
\begin{equation}
  (D_1^2 \mathbf{u}, \mathbf{u}_t) = \mathbf{u}_t^H D_1^2 \mathbf{u}
                                    = (-D_1\mathbf{u}_t)^H D_1 \mathbf{u}
                                    = -(D_1 \mathbf{u}, D_1 \mathbf{u}_t).
\end{equation}
By adding up Eq. (\ref{eqn:addsemi0}) with its conjugation, we obtain
\begin{equation}\label{eqn:lecproofsemi}
  \begin{aligned}
    0 &= -\mathrm{i} (\mathbf{u}_t, \mathbf{u}_t) + \mathrm{i} ({\mathbf{u}}_t, \mathbf{u}_t) \\
    & = -(D_1 \mathbf{u}, D_1 \mathbf{u}_t) -(D_1 \mathbf{u}_t, D_1 \mathbf{u}) + (\beta\mathbf{r}, \bar{\mathbf{u}}_t \cdot \mathbf{u} + {\mathbf{u}}_t \cdot \bar{\mathbf{u}}) \\
    & = 2\frac{d}{dt} (-\frac{1}{2} \|D_1 \mathbf{u}\|^2 + \frac{\beta}{4} \|\mathbf{r}\|^2).
  \end{aligned}
\end{equation}
\end{proof}

\begin{rem}
  It is more common to define the global energy for the discrete system of NLS equation as $\tilde E(\mathbf{u}) = -\frac{1}{2} \|D_1 \mathbf{u}\|^2 + \frac{\beta}{4} \|\mathbf{u}\|^4$. However,
  since the discrete global energy defined in Eq. \eqref{eqn:gclsemi} is just another kind of discretization without any loss of accuracy for Eq. \eqref{eq:init-energy}, $E(\mathbf{u})$ is recognized as
  simply the discrete global energy in the following sections.
\end{rem}

According to Thm. \ref{thm:semi-energy}, we have an important lemma which will play a key role in the proof for the energy conservation law of the fully discrete system.
\begin{lem}\label{lem1}
  $\forall \mathbf{u} \in \mathbb{C}^N, \mathbf{r} \in \mathbb{R}^N$, we have
  \begin{equation}\label{eq:cor1-equ}
    \frac{1}{2} \Re(f(\mathbf{u}, \mathbf{r})^T D_1^2 \bar{\mathbf{u}}) + \frac{\beta}{4} g(\mathbf{u}, \mathbf{r})^T \mathbf{r} = 0.
  \end{equation}
\end{lem}
\begin{proof}
\begin{equation}
  \begin{split}
    \frac{\mathrm{d}}{\mathrm{d} t}E(\mathbf{u}) &= \frac{\mathrm{d}}{\mathrm{d} t} (-\frac{1}{2}\|D_1 \mathbf{u}\|^2 + \frac{\beta}{4} \|\mathbf{r}\|^2) \\
                &= \frac{1}{2} ( \mathbf{u}_t^H D_1^2 \mathbf{u} + \mathbf{u}^H D_1^2 \mathbf{u}_t ) + \frac{\beta}{2} \mathbf{r}_t^T \mathbf{r} \\
                &= \frac{1}{2}(\Re(f(\mathbf{u}, \mathbf{r})^T D_1^2 \bar{\mathbf{u}} ) + \frac{\beta}{4} g(\mathbf{u}, \mathbf{r})^T \mathbf{r})\\
                &= 0.
  \end{split}
\end{equation}
\end{proof}

\begin{thm}\label{thm:semi-mass}
  The semi-discrete system \eqref{eqn:fnsesemi1} -- \eqref{eqn:fnsesemi3} admits the semi-discrete mass conservation law:
  \begin{equation}\label{eq:semi-mass}
    \frac{\mathrm{d}}{\mathrm{d} t}M(\mathbf{u}, \mathbf{r}) = \frac{\mathrm{d}}{\mathrm{d} t} \|\mathbf{u}\|^2 = 0.
  \end{equation}
\end{thm}
\begin{proof}
  Taking the inner product of \eqref{eqn:fnsesemi1} with $\mathbf{u}$ gives
  \begin{equation}\label{semi-mass-1}
  \begin{aligned}
   ({\mathbf{u}}_t, \mathbf{u}) &= \mathrm{i}(D_1^2 \mathbf{u}, \mathbf{u}) + \mathrm{i}(\beta\mathbf{r} \cdot \mathbf{u}, \mathbf{u}) \\
    & = -\mathrm{i}\|D_1 \mathbf{u}\|^2 + \mathrm{i}(\beta\mathbf{r}, \mathbf{u} \cdot \bar{\mathbf{u}}).
  \end{aligned}
\end{equation}
And then by taking the real part of \eqref{semi-mass-1} we obtain
\begin{equation}
\label{eq:semi-mass-2}
  \frac{\mathrm{d}}{\mathrm{d} t}M(\mathbf{u}, \mathbf{r}) = 0.
\end{equation}
\end{proof}

Next we study the two-dimensional NLS equation:
\begin{equation}
\label{2dim-nls}
  iu_t+u_{xx}+u_{yy}+\beta|u|^2u=0.
\end{equation}
on bounded domain $\Omega = [a, b] \times [c, d]$ with inital and periodic boundary conditions
\begin{align}
    & u(x, y, 0) =u_{0}(x, y), \ (x, y) \in \Omega, \label{2dim-init} \\
    & u(a, y, t) = u(b, y, t), \ u(x, c, t) = u(x, d, t), \ t \geq 0. \label{2dim-bound}
\end{align}
Singularities is possibly developed for the 2D NLS equation \cite{glassey1977blowing} at finite time. With $r=|u|^2$, Eq. \eqref{2dim-init} can be transformed into

\begin{align}
    u_t &= \mathrm{i}u_{xx} + \mathrm{i}u_{yy}+ \mathrm{i}\beta r u, \label{2dim-ieq1} \\
    r_t &= u\bar{u}_t + \bar{u} u_t \label{2dim-ieq2}.
\end{align}

Let $N_x$ and $N_y$ denote the number of the subintervals on $x$ and $y$ direction,
respectively. Without any loss of generality, let $N_x=N_y=N$. Using the Fourier pseudospectral method, the semi-discrete system of \eqref{2dim-init} -- \eqref{2dim-ieq2} is given by:
\begin{align}
  \mathbf{u}_t &= \mathrm{i} D_1^2\mathbf{u}^{n} + \mathrm{i} \mathbf{u}^{n} (D_1^2)^T + \beta\mathbf{r} \cdot \mathbf{u}, \label{eq:2dim-semi1} \\
  \mathbf{r}_t &= \bar{\mathbf{u}} \cdot  \mathbf{u}_t + \mathbf{u} \cdot  \bar{\mathbf{u}}_t. \label{eq:2dim-semi2}
\end{align}

With the definition of inner product $(\mathbf{u}, \mathbf{v}) = h^2\sum_{i, j=1}^Nu_{i,j}\bar{v}_{i,j}$ and norm $\|\mathbf{u}\|^2 = (\mathbf{u}, \mathbf{v}) $, the discrete mass and energy for 2D cases are similarly denoted by
\begin{align}
  M(\mathbf{u}) &= \|\mathbf{u}\|^2, \label{eq:2dim-mass} \\
  E(\mathbf{u}) &= -\frac{1}{2}(\|D_1\mathbf{u}\|^2 + \|\mathbf{u}D_1^T\|^2) + \frac{\beta}{4} \|\mathbf{r}\|^2. \label{eq:2dim-energy}
\end{align}

Similar to the proof in Thm \ref{thm:semi-mass} and \ref{thm:semi-energy}, we can prove that for the 2D NLS equation the conservation laws for mass and energy still hold.

\section{High order conservative diagonally implicit IEQ--RK method}\label{sec:4}
In this section, by applying the technique of invariant energy quadratization, a class of high order global mass and energy conserving diagonally implicit Runge--Kutta schemes are presented. We first recall
the $s$-stage RK scheme.

\textit{$s$-stage RK scheme}. Let $b_i, a_{i j} (i, j = 1, \cdots s)$ be real numbers and $c_i = \sum_{j = 1}^{s} a_{i j}$. $\Delta t$ denotes the time step, $t_n=n\Delta t$ and $\mathbf{u}^n$ and
$\mathbf{r}^n$ stand for the numerical solution for $u$ and $r$ on time $t_n$. For given $\mathbf{u}^n, \mathbf{r}^n$, the following intermediate values are calculated by
\begin{align}
  \mathbf{U}_i &= \mathbf{u}^n + \Delta t \sum_{j = 1}^s a_{i j} \mathbf{f}_j, \label{eqn:rk-mid-1}\\
  \mathbf{R}_i &= \mathbf{r}^n + \Delta t \sum_{j = 1}^s a_{i j} \mathbf{g}_j, \label{eqn:rk-mid-2}
\end{align}

where $\mathbf{f}_i = f(\mathbf{U}_i, \mathbf{R}_i)$ , $\mathbf{g}_i = g(\mathbf{U}_i, \mathbf{R}_i)$. Then $\mathbf{u}^{n+1}$ and $\mathbf{r}^{n+1}$ are updated by
\begin{align}
  \mathbf{u}^{n+1} &= \mathbf{u}^n + \Delta t \sum_{i = 1}^s b_i \mathbf{f}_i, \label{eqn:rk-1}\\
  \mathbf{r}^{n+1} &= \mathbf{r}^n + \Delta t \sum_{i = 1}^s b_i \mathbf{g}_i. \label{eqn:rk-2}
\end{align}

Here we only present the proof for the energy and mass conservative properties of 1D discrete system. However, readers will see it is trivial to generalize these proofs to the 2D cases.
\begin{thm}
  The scheme \eqref{eqn:rk-1} -- \eqref{eqn:rk-2} preserves the discrete global energy exactly, namely
  \begin{equation}\label{eq:energy-conservative}
    E(\mathbf{u}^{n+1}, \mathbf{r}^{n+1}) = E(\mathbf{u}^{n}, \mathbf{r}^{n}),
  \end{equation}
  if and only if the symmetric coefficient matrix $M$ with elements
  \begin{equation}\label{eq:M-matrix}
    m_{i j} = b_i a_{i j} + b_j a_{j i} - b_i b_j, \ i, j = 1, ..., s.
  \end{equation}
  is the zero matrix.
\end{thm}

\begin{proof}
Eqns. \eqref{eqn:rk-mid-1} and \eqref{eqn:rk-1} yield
\begin{equation}\label{eq:u-between-steps}
  \begin{split}
    -\frac{1}{2} \| D_1 \mathbf{u}^{n+1} \|^2 =& - \frac{1}{2}(D_1(\mathbf{u}^n + \Delta t \sum_{i = 1}^s b_i \mathbf{f}_i),D_1(\mathbf{u}^n + \Delta t \sum_{j = 1}^s b_j \mathbf{f}_j)) \\
                                            =& \frac{1}{2} (-\| D_1 \mathbf{u}^{n} \|^2 + h\Delta t \sum_{i = 1}^s b_i \mathbf{f}_i^H D_1^2 (\Delta t \sum_{j = 1}^s a_{i j} \mathbf{f}_j - \mathbf{U}_i) + h\Delta t \sum_{i = 1}^s (\Delta t \sum_{j = 1}^s a_{i j} \mathbf{f}_j^H - \mathbf{U}_i^H) D_1^2 b_i \mathbf{f}_i \\
                                                &+ h\Delta t^2 \sum_{i, j = 1}^s b_i b_j \mathbf{f}_i^H D_1^2 \mathbf{f}_j) \\
                                            =& \frac{1}{2} (-\| D_1 \mathbf{u}^{n} \|^2 + h\Delta t^2 \sum_{i, j = 1}^s (b_i a_{i j} + b_j a_{j i} - b_i b_j) \mathbf{f}_i^H D_1^2 \mathbf{f}_j - h\Delta t
                                            \sum_{i = 1}^s b_i (\mathbf{f}_i^H D_1^2 \mathbf{U}_i + \mathbf{f}_i^T D_1^2 \overline{\mathbf{U}_i}).
  \end{split}
\end{equation}
Similarly, Eqns. \eqref{eqn:rk-mid-2} and \eqref{eqn:rk-2} give
\begin{equation}\label{eq:r-between-steps}
    \frac{\beta}{4} \| \mathbf{r}^{n+1} \|^2 = \frac{\beta}{4} (\| D_1 \mathbf{r}^{n} \|^2 + h\Delta t^2 \sum_{i, j = 1}^s (b_i a_{i j} + b_j a_{j i} - b_i b_j) \mathbf{g}_i^T D_1^2 \mathbf{g}_j - 2 h\Delta t \sum_{i = 1}^s b_i \mathbf{g}_i^T D_1^2 \mathbf{R}_i).
\end{equation}
From Lem. \ref{lem1}, we have
\begin{equation}\label{eq:lem1-1}
  \frac{1}{2} ( \mathbf{f}_i^H D_1^2 \mathbf{U}_i + \mathbf{f}_i^T D_1^2 \overline{\mathbf{U}_i}) + \frac{\beta}{2} \mathbf{g}_i^T D_1^2 \mathbf{R}_i = 0.
\end{equation}
Then by substituting Eq. \eqref{eq:lem1-1} into the sum of \eqref{eq:u-between-steps} with \eqref{eq:r-between-steps} we have
\begin{equation}\label{eq:energy-mid}
  E(\mathbf{u}^{n+1}, \mathbf{r}^{n+1}) = E(\mathbf{u}^{n}, \mathbf{r}^{n}) + h\Delta t^2 \sum_{i, j = 1}^s  (b_i a_{i j} + b_j a_{j i} - b_i b_j)(\mathbf{f}_i^H D_1^2 \mathbf{f}_j + \mathbf{g}_i^T D_1^2 \mathbf{g}_j).
\end{equation}
Finally, substituting Eq. \eqref{eq:M-matrix} into \eqref{eq:energy-mid} gives \eqref{eq:energy-conservative}.

Meanwhile, the proof of necessity is easily drown from Eq. \eqref{eq:energy-mid}. So the proof is completed.
\end{proof}

\begin{rem}
  The energy conservation law \eqref{eq:energy-conservative} is a so-called $\mathbf{S}$-conservative property:
\begin{equation}
  \label{eq:3}
  \begin{pmatrix} (\mathbf{u}^{n+1})^{H} & (\mathbf{r}^{n+1})^T \end{pmatrix} \mathbf{S}
  \begin{pmatrix} \mathbf{u}^{n+1} \\ \mathbf{r}^{n+1} \end{pmatrix} =
  \begin{pmatrix} (\mathbf{u}^n)^H & (\mathbf{r}^n)^T \end{pmatrix} \mathbf{S}
    \begin{pmatrix} \mathbf{u}^n \\ \mathbf{r}^n \end{pmatrix},
\end{equation}
where $\mathbf{S}$ is a real symmetric matrix denoted by
\begin{equation}
\label{eq:4}
\mathbf{S} = \begin{pmatrix}
  \frac{1}{2}D_1^2 & O\\
  O & \frac{\beta}{4}I_N
  \end{pmatrix}.
\end{equation}
\end{rem}

\begin{thm}
  Scheme \eqref{eqn:rk-1} -- \eqref{eqn:rk-2} preserves the discrete mass exactly, namely
  \begin{equation}\label{eq:mass}
    M(\mathbf{u}^{n+1}, \mathbf{r}^{n+1}) = M(\mathbf{u}^{n}, \mathbf{r}^{n}),
  \end{equation}
  if and only if the symmetric coefficient matrix $M$ is the zero matrix.
\end{thm}
\begin{proof}
Substituting the matrix $D_1^2$ in Eq. \eqref{eq:u-between-steps} with $I_N$ gives
\begin{equation}
  \label{eq:mass-1}
  \| \mathbf{u}^{n+1} \|^2 =  \| \mathbf{u}^{n} \|^2 + h\Delta t^2 \sum_{i, j = 1}^s (b_i a_{i j} + b_j a_{j i} - b_i b_j) \mathbf{f}_i^H \mathbf{f}_j - h\Delta t \sum_{i = 1}^s b_i (\mathbf{f}_i^H \mathbf{U}_i +
  \mathbf{f}_i^T \overline{\mathbf{U}_i}).
\end{equation}
Combining Eq. \eqref{eq:mass-1} with \eqref{eq:semi-mass-2} yields
\begin{equation}
\label{eq:mass-2}
  \| \mathbf{u}^{n+1} \|^2 =  \| \mathbf{u}^{n} \|^2 + h\Delta t^2 \sum_{i, j = 1}^s m_{i j} \mathbf{f}_i^H \mathbf{f}_j,
\end{equation}
which complete the proof.
\end{proof}
With $\tilde {\mathbf{S}}=I_N$, the discrete mass conservaion law is another $\tilde {\mathbf{S}}$-conservative law that scheme \eqref{eqn:rk-1} -- \eqref{eqn:rk-2} admits.
\begin{thm}\rm{(\cite{feng2010symplectic})}
The diagonally implicit Runge--Kutta (DIRK) schemes satisfying the condition \eqref{eq:M-matrix} can be expressed in Butcher notation by the table of coefficients,
\begin{equation}
	\label{eqn:dirk}
	\begin{array}{c|c}
		c & A \\
		\hline \\
		& b^T
	\end{array}
	=
	\begin{array}{c|cccc}
		c_1 & \frac{b_1}{2} &  						  &  &\\
		c_2 & b_1 					& \frac{b_2}{2} &  & \\
		\vdots& \vdots 				& \vdots 				& \ddots & \\
		c_s & b_1 				  & b_2 				  & \cdots & \frac{b_s}{2} \\
		\hline\\
		&b_1 & b_2 		& \cdots 				& b_s
	\end{array}
\end{equation}
where $c_i = \sum_{j = 1}^s a_{i j}$, and $b_i \neq 0$ ($i = 1, 2, \cdots, s$), $a_{ij} = 0 (i<j)$.
\end{thm}
\textit{$s$-stage IEQ-DIRK scheme}. By applying this series of coefficients the discrete system can be transformed into diagonally implicit Runge--Kutta schemes
\begin{align}
  \mathbf{U}_i = \mathbf{u}^n + \Delta t \sum_{j = 1}^i a_{i j} \mathbf{f}_j \label{eqn:rk-ieq-1},\\
  \mathbf{R}_i = \mathbf{r}^n + \Delta t \sum_{j = 1}^i a_{i j} \mathbf{g}_j \label{eqn:rk-ieq-2},
\end{align}
and each pair of $(\mathbf{U}_i \ \mathbf{R}_i)$ are calculated in order, which considerably contribute to the effciency of the IEQ-RK schemes.

 The quadratic invariant-preserving Runge--Kutta schemes have been further discussed by many papers \cite{sanz1991order,franco2003fourth,kalogiratou2012diagonally}. Below we present several DIRK schemes that satisfies \eqref{eqn:dirk} and will be used in our work.

\begin{table}[h!]
	\caption{The value of $b_i$ for different stages of DIRK satisfying \eqref{eqn:dirk}.}
\begin{center}
\begin{tabular}{c|ccc}
	\hline
	\raisebox{-1.6ex}[0pt]{(Stage, convergence order)} & $b_1$ & $b_2$ & $b_3$  \\
	& $b_4$ & $b_5$ & $b_6$ \\  \hline
 (1, 2)\cite{feng2010symplectic} & $1$ & & \\ \hline
 (2, 2)\cite{feng2010symplectic} & $\frac{1}{2}$ & $\frac{1}{2}$ &  \\ \hline
 (3, 3)\cite{feng2010symplectic} & $1.351207$ & $1.351207$ & $-1.702414$  \\ \hline
 \raisebox{-1.6ex}[0pt]{(4, 4)\cite{feng2010symplectic}} & $2.70309412$ & $-0.53652708$ & $2.37893931$  \\
 & $1.8606818856$ & & \\  \hline
 \raisebox{-1.6ex}[0pt]{(5, 4)\cite{kalogiratou2012diagonally}} & $-2.150611289942181$ & $1.452223059167718$ & $2.3967764615489258$ \\
 & $1.452223059167718$ & $-2.150611289942181$ &  \\  \hline
 \raisebox{-1.6ex}[0pt]{(6, 5)\cite{kalogiratou2012diagonally}} & $0.5080048194000274$ & $1.360107162294827$ & $2.0192933591817224$  \\
 & $0.5685658926458251$ & $-1.4598520495864393$ & $-1.9961191839359627$ \\
 \hline
\end{tabular}
\end{center}
\label{tab:butcher}
\end{table}

\section{Numerical experiment}\label{sec:5}
\textbf{Example 5.}(Travelling wave solution) In this experiment, we verify the convergence orders and testify the mass and energy conservative properties of proposed schemes. The initial condition is
\begin{equation}
  u_0 = \sech(x)\exp(2\mathrm{i}x).
\end{equation}
With $\beta = 2$, this system admits an exact single-soliton solution
\begin{equation}
  u(x, t) = \sech(x-4t)\exp(\mathrm{i}(2x-3t)).
\end{equation}
First, let $N = 256$, $T =3$ and $\Delta t=0.01$. The profiles of numerical solutions obtained by IEQ-DIRK(2, 2) and IEQ--DIRK(4, 4) are demonstrated in Fig. \ref{fig:schr-profiles} as examples, along with the real solution
for comparison. It is shown that both the low and high order IEQ-DIRK schemes simulate Eqns. \eqref{eqn:nlse} -- \eqref{boundary_condition} accurately. In addition, the mass and energy errors of
different stages IEQ-DIRK schemes are presented
in Fig. \ref{fig:schr-mass-energy}, which are all controlled in a same extremely low magnitude, from which we conclude that the proposed IEQ-DIRK
schemes ranging from 2nd to 5th order preserve these two invariants perfectly, since the inaccuracy is negligible (less than $10^{-14}$) and caused by the lower bound of iteration errors.

\begin{figure}[htbp]
\centering
\subfloat[real]{\includegraphics[width=.32\textwidth]{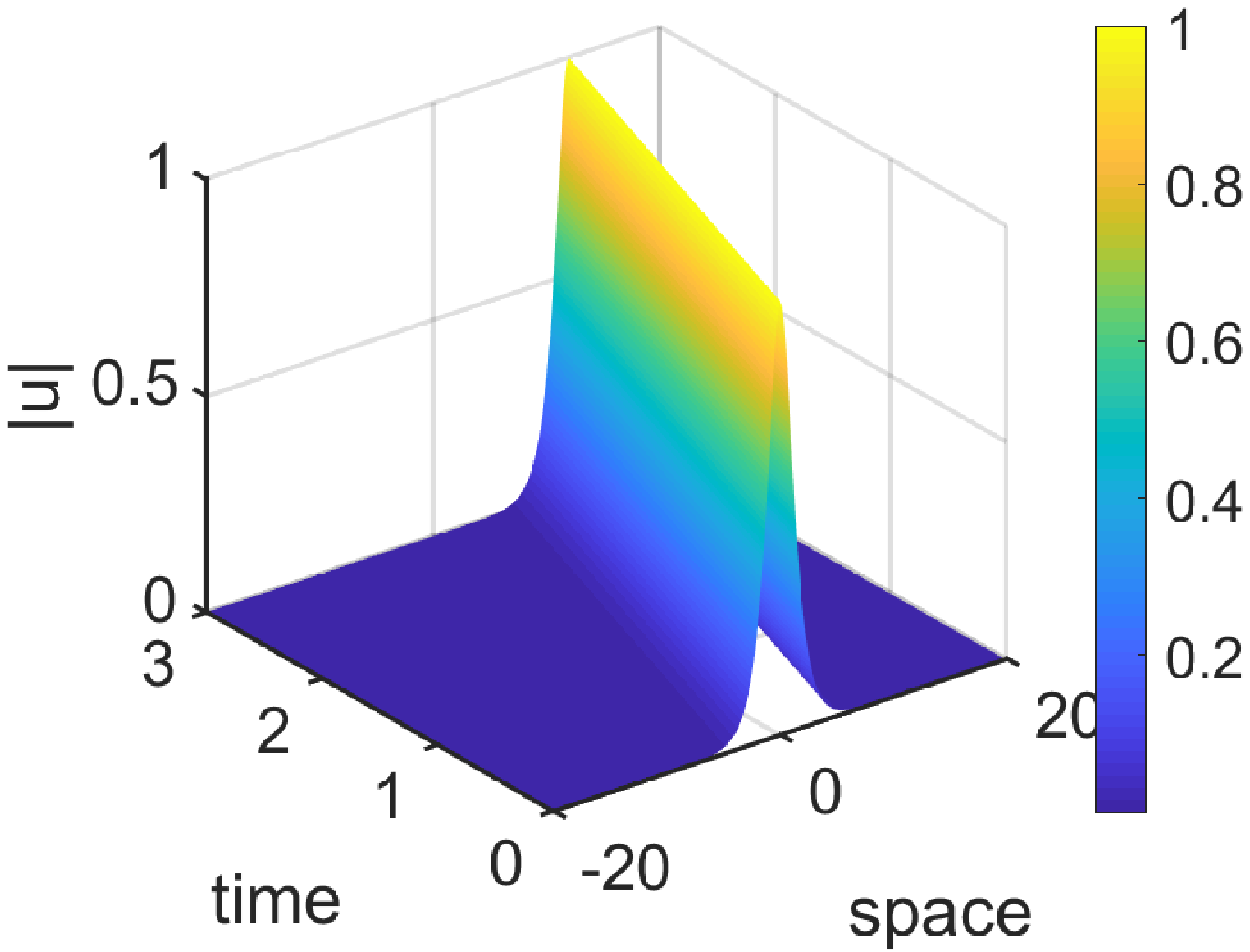}}
\subfloat[RK(2,2)]{\includegraphics[width=.32\textwidth]{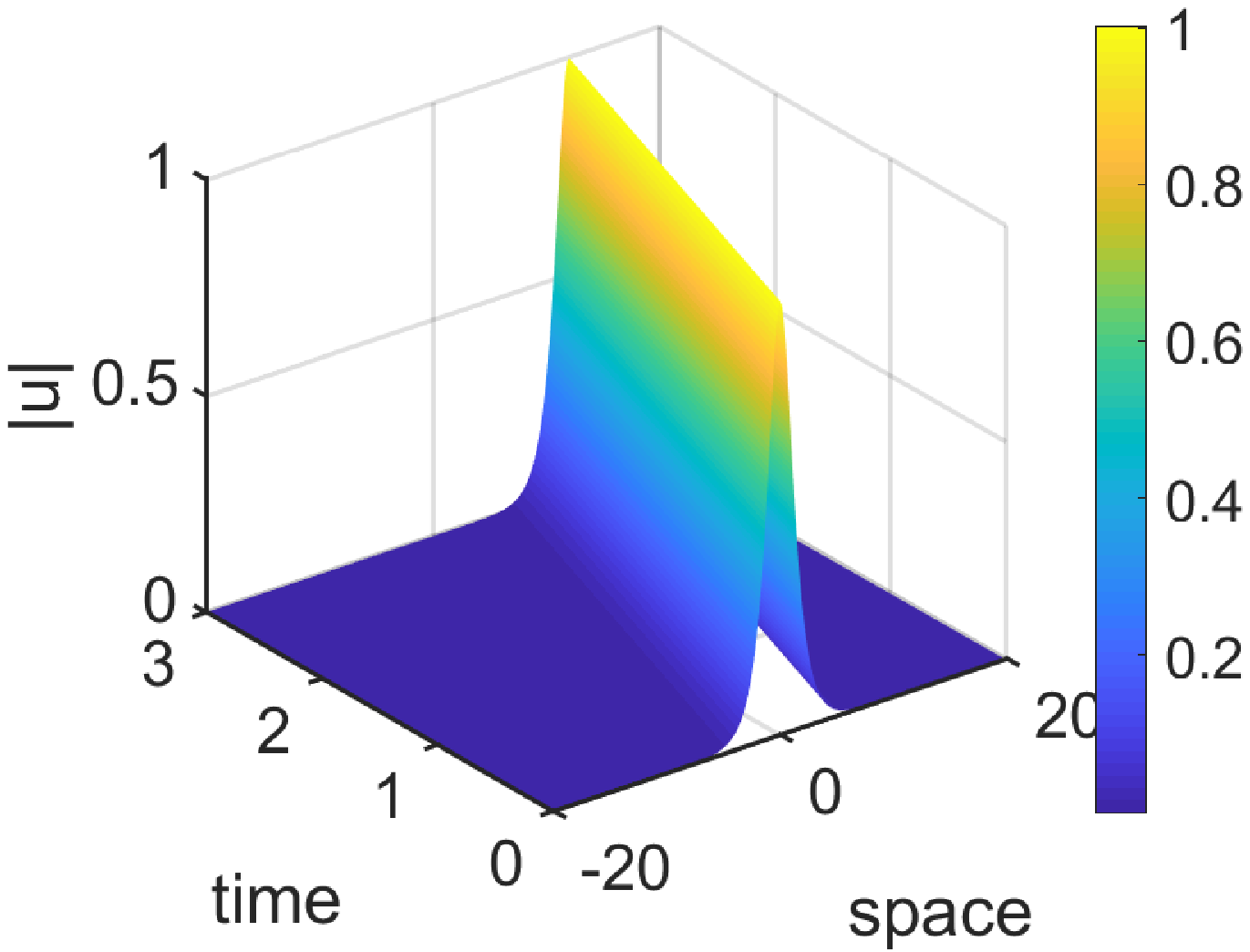}}
\subfloat[RK(4,4)]{\includegraphics[width=.32\textwidth]{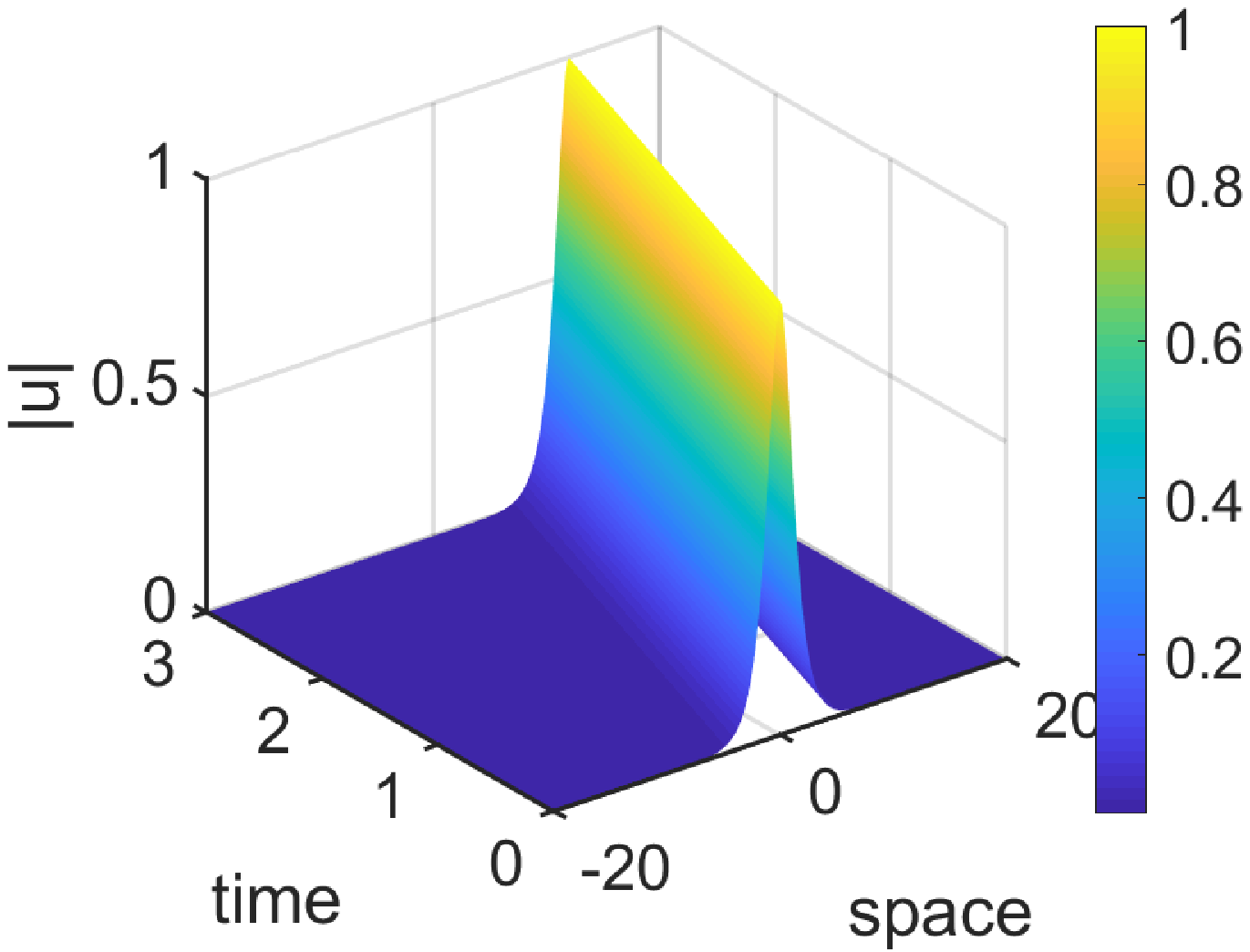}}
\caption{The profiles of real solution, and of numerical solutions of IEQ-DIRK(2, 2) and IEQ--DIRK(4, 4).}
\label{fig:schr-profiles}
\end{figure}
\begin{figure}[htbp]
\centering
\subfloat[mass error]{\includegraphics[width=.48\textwidth]{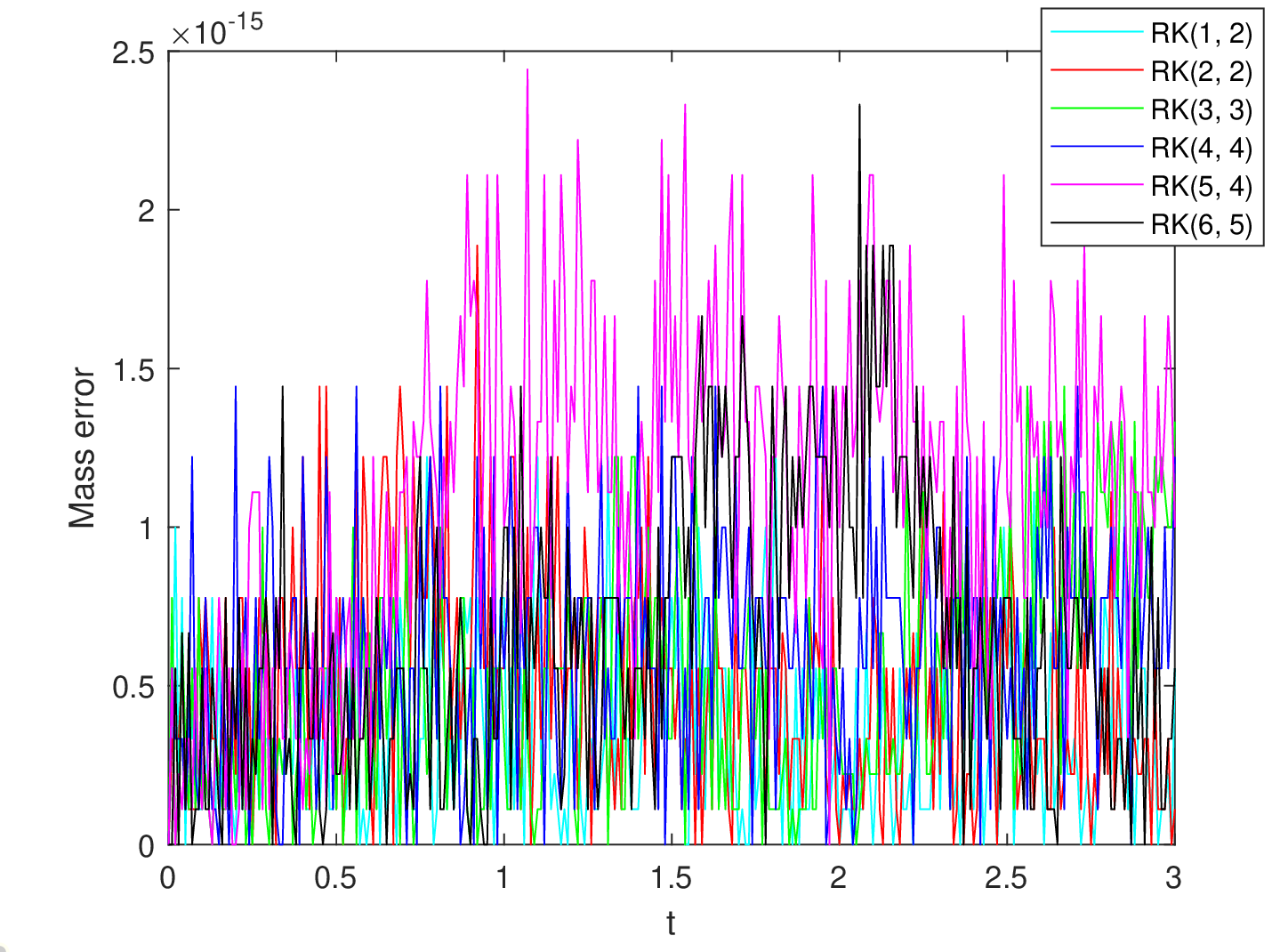}}
\subfloat[modified energy error]{\includegraphics[width=.48\textwidth]{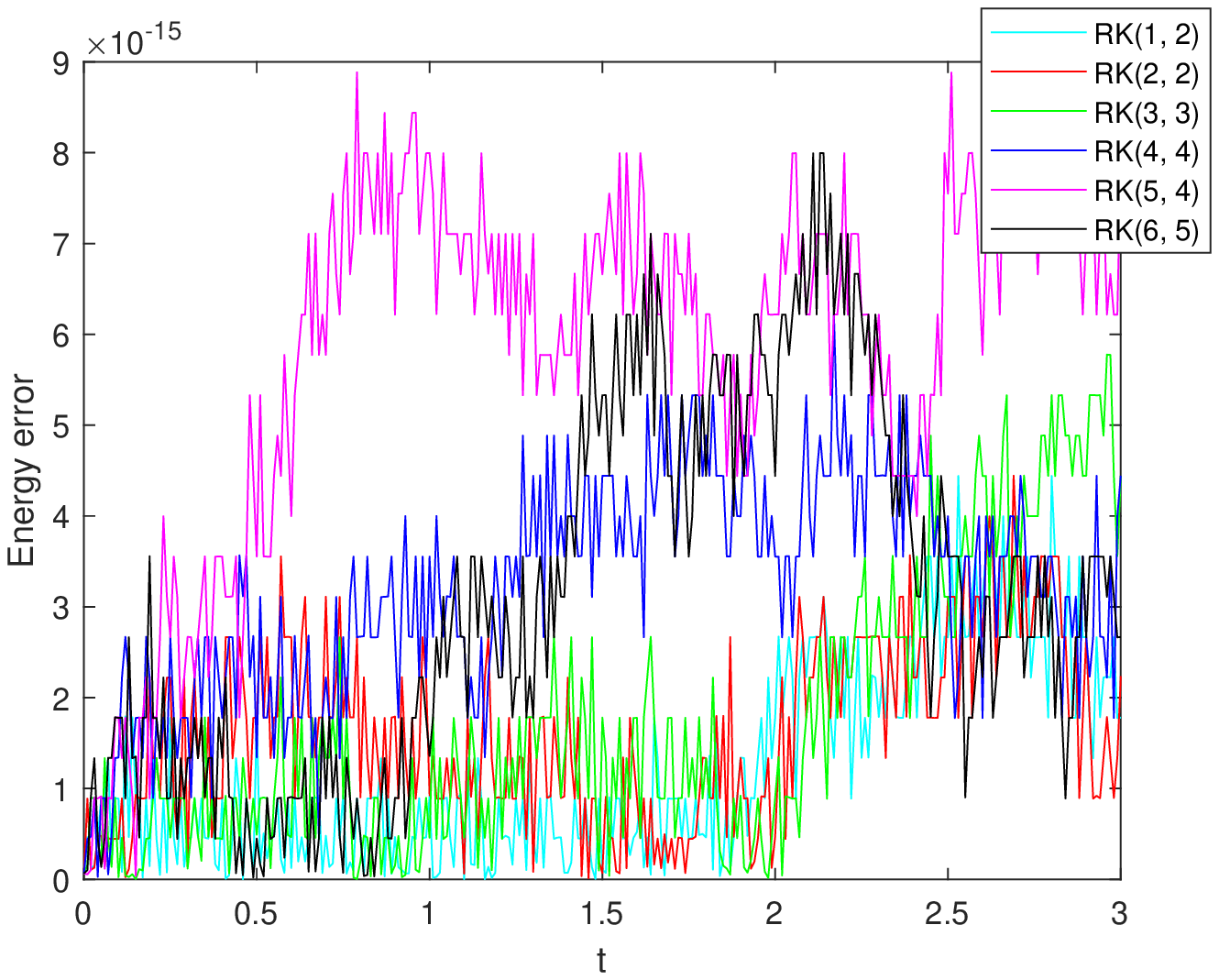}}
\caption{mass and energy error of different stages.}
\label{fig:schr-mass-energy}
\end{figure}

In order to testify the the the $L^2$ convergence order of the IEQ-DIRK schemes in the time direction, we fix $N = 256$, $T = 2^{-5}$, and the $\Delta t$ ranging from $2^{-9}$ to $2^{-6}$. The results are presented in a
log-log plot in Fig. \ref{fig:schr-time-order}. According to most of the samples, with coefficients given in Tab. \ref{tab:butcher}, the numerical convergence orders of IEQ-DIRK schemes agree with the theoretical orders.
\begin{figure}[htbp]
\centering
    {\includegraphics[width=.48\textwidth]{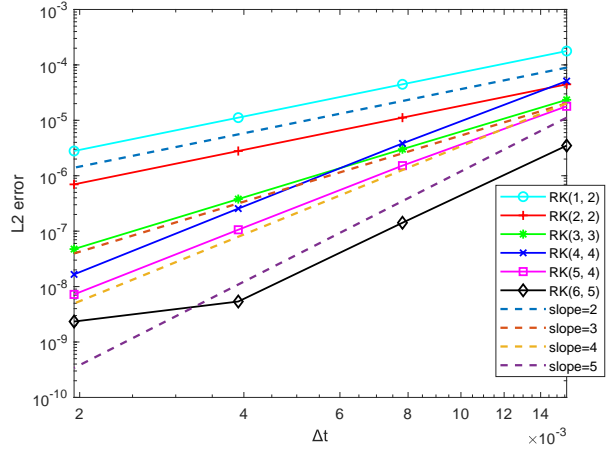}}
\caption{$L^2$ errors of IEQ-DIRK schemes.}
\label{fig:schr-time-order}
\end{figure}

However, the longtime simulating stability is still not demonstrated due to the short fixed time. Correspondingly, we set the initial condition as
\begin{equation}
  \label{eq:6}
   u_0 = \frac{1}{\sqrt{2}}\sech(\frac{1}{\sqrt{2}}(x-25))\exp(-\mathrm{i}\frac{x}{20}).
 \end{equation}
 and let $T = 1000$, $\Delta t=0.01$. The profiles and numerical errors of mass and energy for the computed by IEQ-DIRK(2, 2) and IEQ--DIRK(4, 4) are presented in Fig. \ref{fig:longtime}, which can be regarded as a proof of the rubust property of IEQ-DIRK scheme in longtime simulation.

\begin{figure}[htbp]
  \centering
\subfloat[profile]{\includegraphics[width=.32\textwidth]{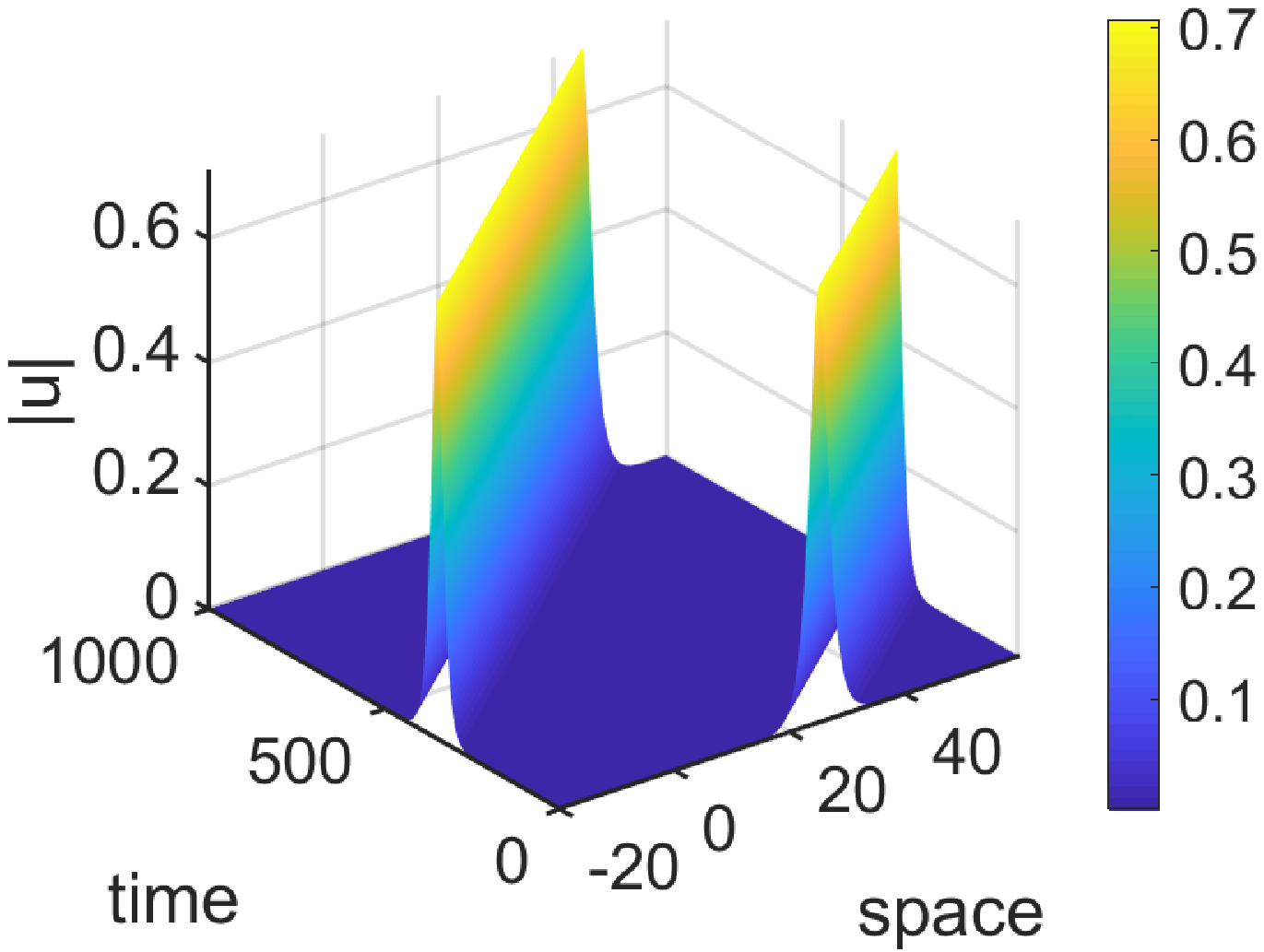}}
\subfloat[mass error]{\includegraphics[width=.32\textwidth]{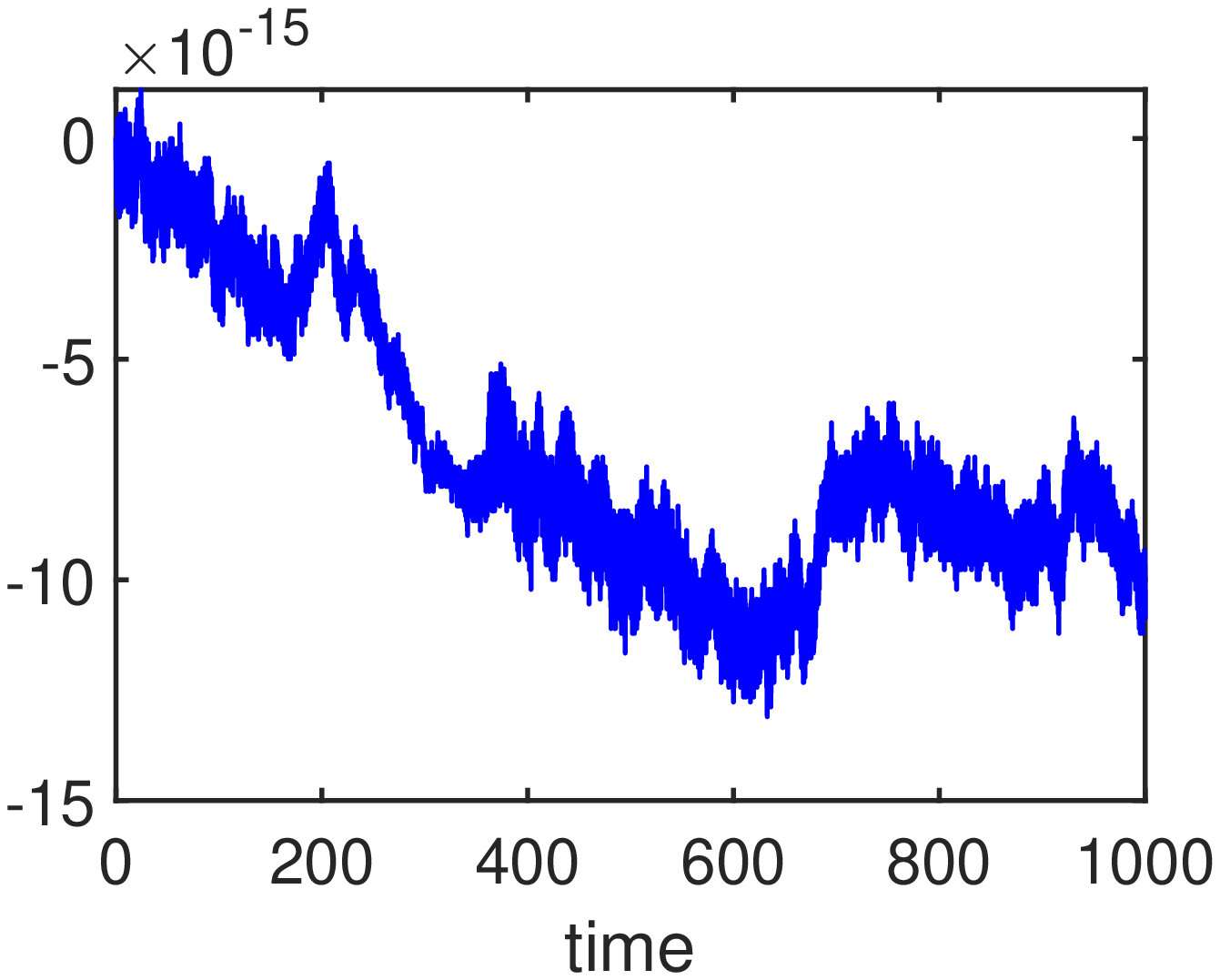}}
\subfloat[energy error]{\includegraphics[width=.32\textwidth]{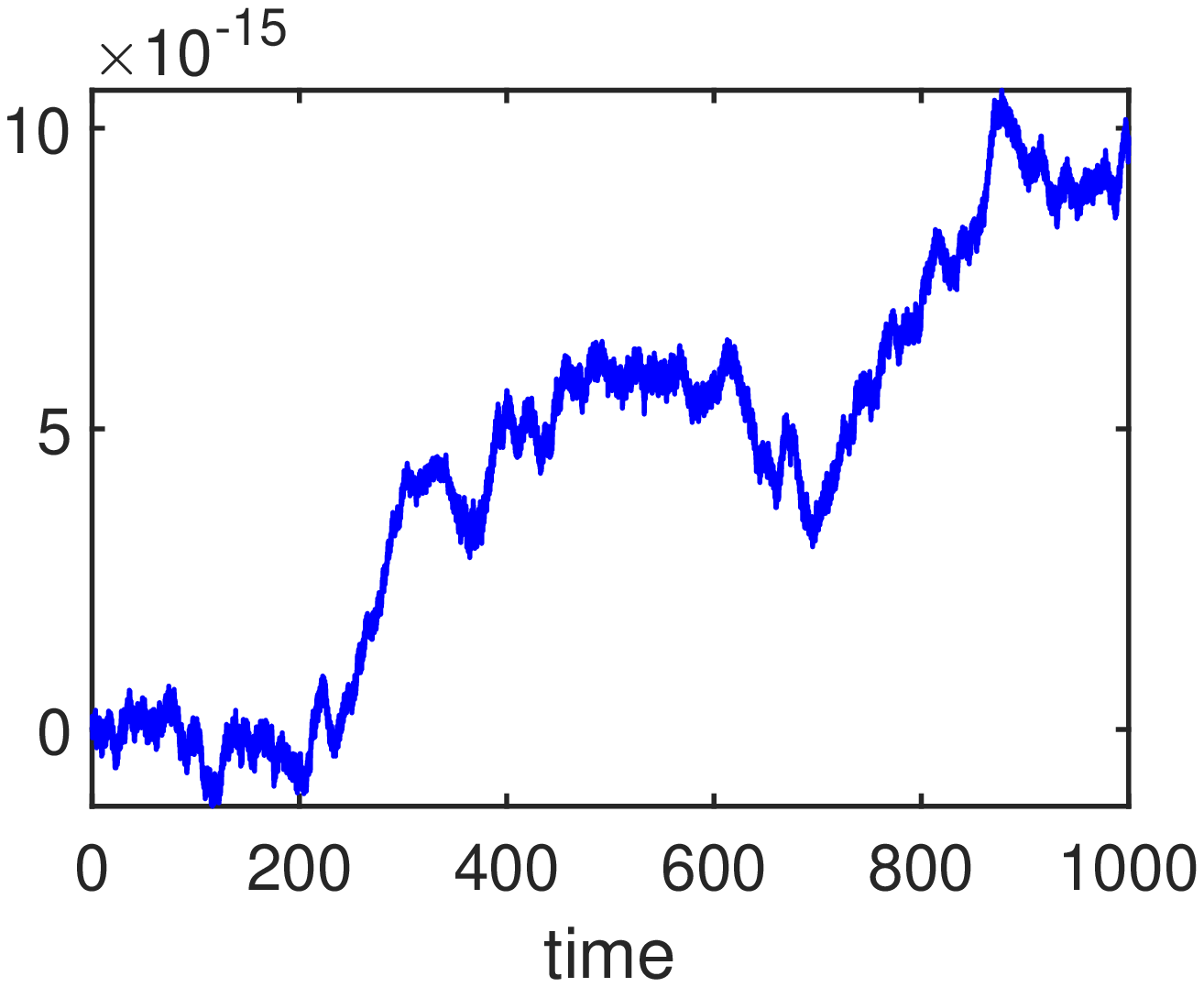}}

\subfloat[profile]{\includegraphics[width=.32\textwidth]{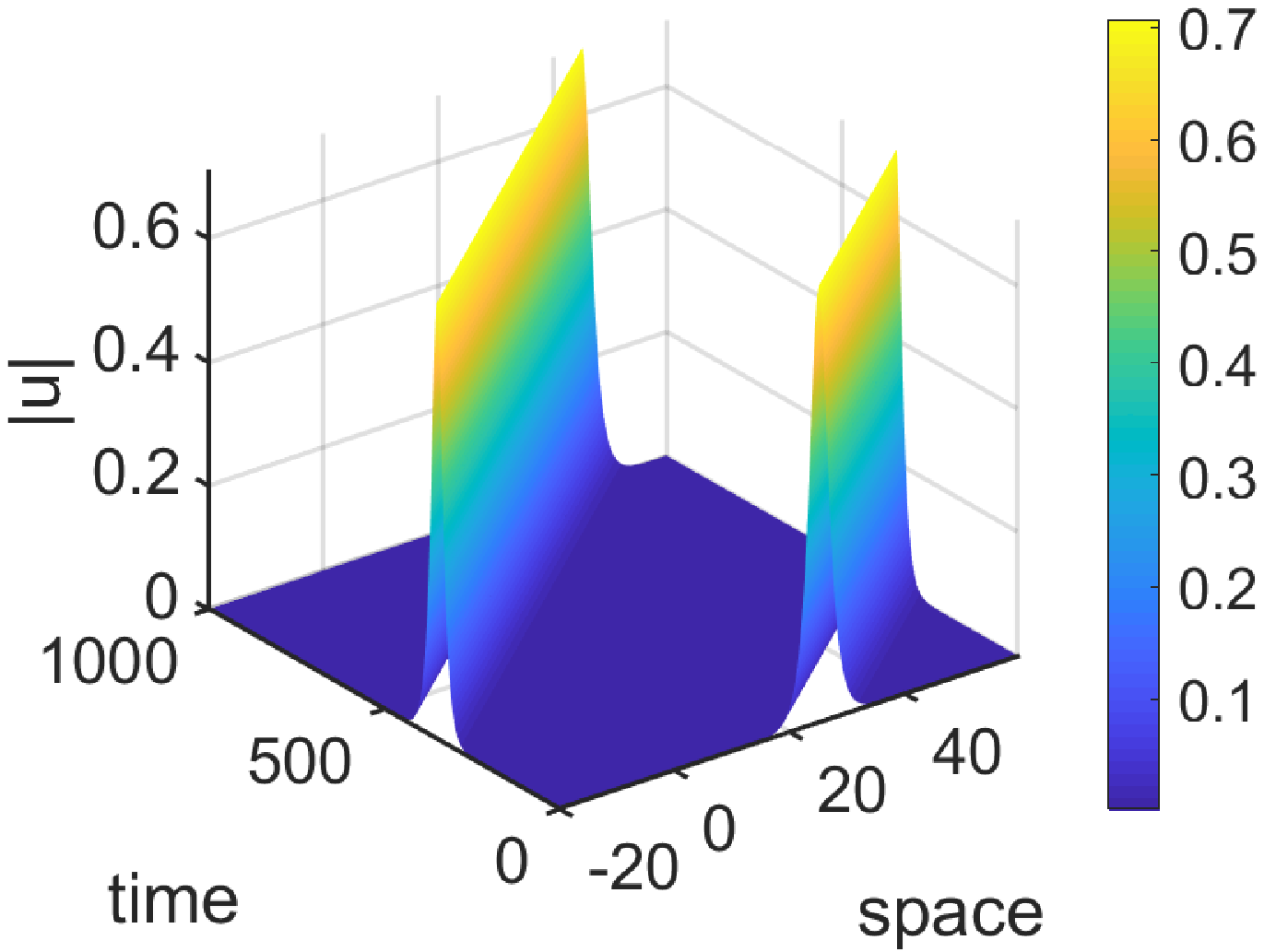}}
\subfloat[mass error]{\includegraphics[width=.32\textwidth]{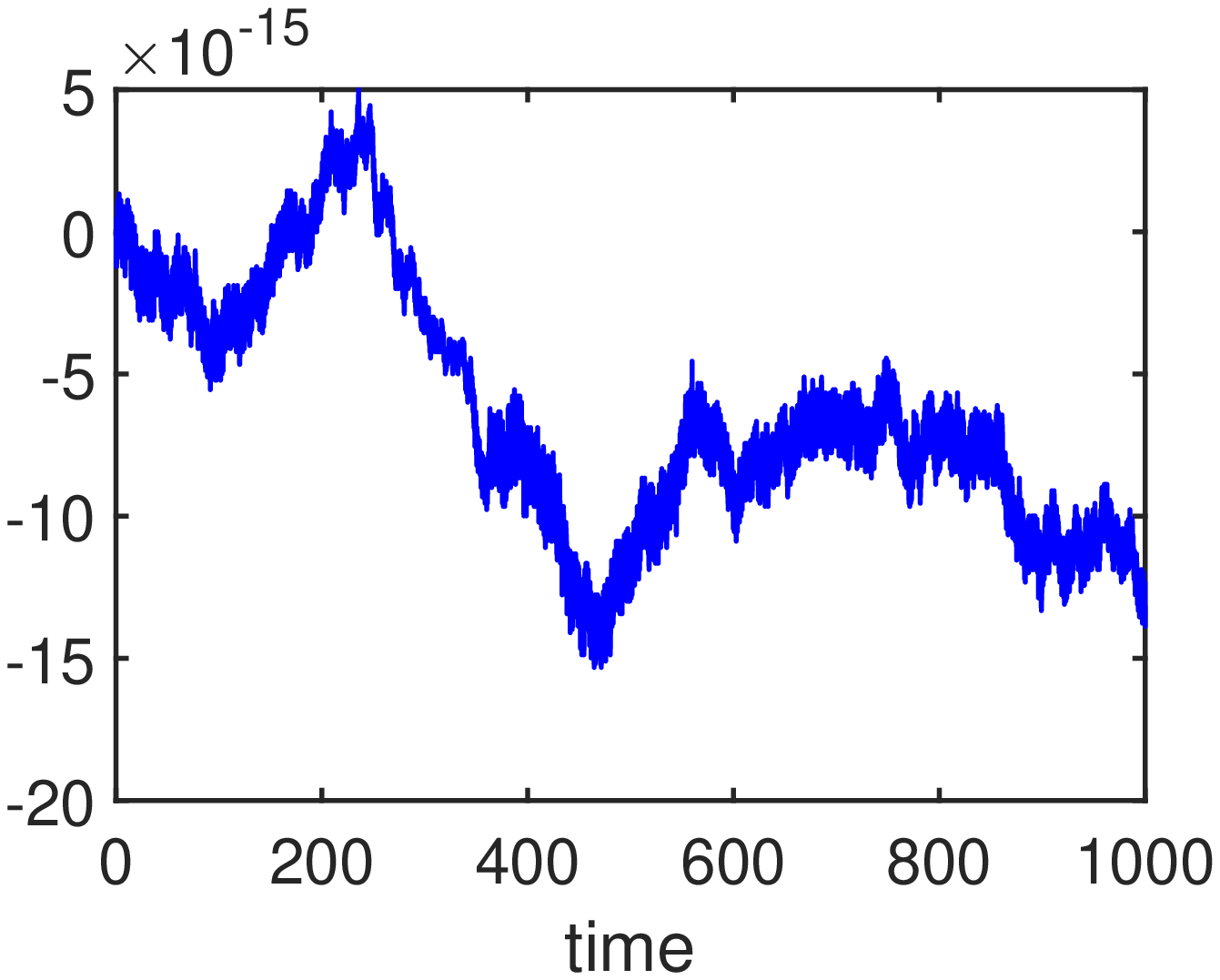}}
\subfloat[energy error]{\includegraphics[width=.32\textwidth]{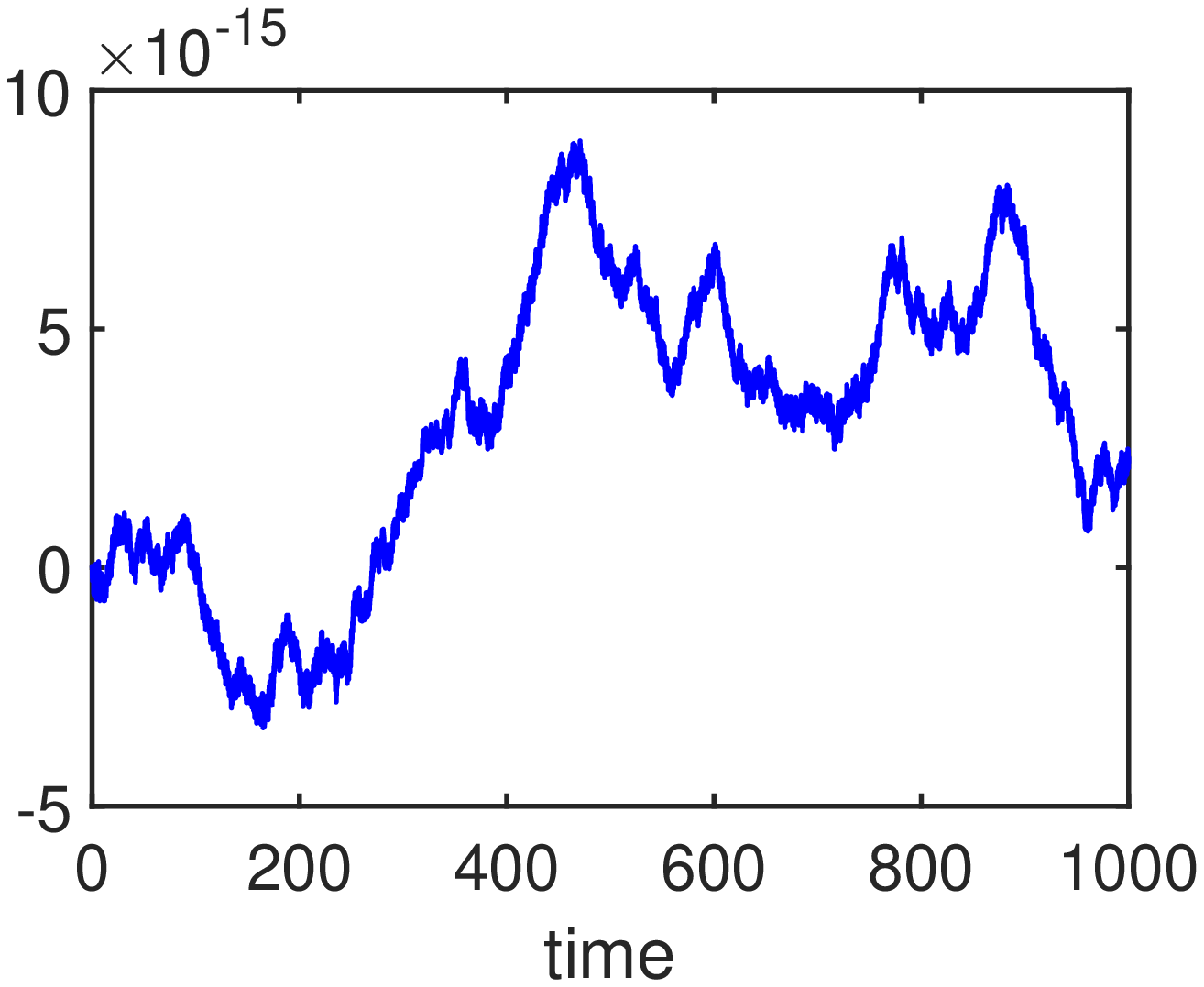}}
\caption{The profiles and numerical errors of mass and energy in longtime simulation by IEQ-DIRK(2, 2) (upper) and IEQ-DIRK(4, 4) (lower).}
  \label{fig:longtime}
\end{figure}

\textbf{Example 2.}(Singular solution in 2D danamics)
In this example, we consider the 2D NLS equation \eqref{2dim-nls} -- \eqref{2dim-bound} with
\begin{equation}
\label{exp:2dim-init}
  u_0(x, y) = (1+\sin x)(2+\sin y), (x, y) \in [0, 2\pi] \times [0, 2\pi].
\end{equation}
Let $\beta = 1$, $N_x = N_y = 128$ and $\Delta t = 0.0001$. We choose IEQ-DIRK(3, 3) to compute the numerical solution. The profiles and corresponding contours at $t = 0$ and $t = 0.108$ are
presented in Figs. \ref{fig:2d-bgn} and \ref{fig:2d-end}. The result of the singular solution is similar with
that obtained by discontinuous Galerkin method \cite{xu2005local}, and shows that the
singularities are captured very well. What's more, although there is a rapid increase on $|u_t|$ when singularities appear, the mass and energy errors shown in Fig. \ref{fig:2d-error} are still accurately controlled when compared
with those of
other non-energy-conserving methods \cite{zhu2011symplectic}. It is a cogent evidence for the conservative properties of IEQ--DIRK schemes.

\begin{figure}[htbp]
  \centering
\subfloat[the profile at $t = 0$]{\includegraphics[width=.48\textwidth]{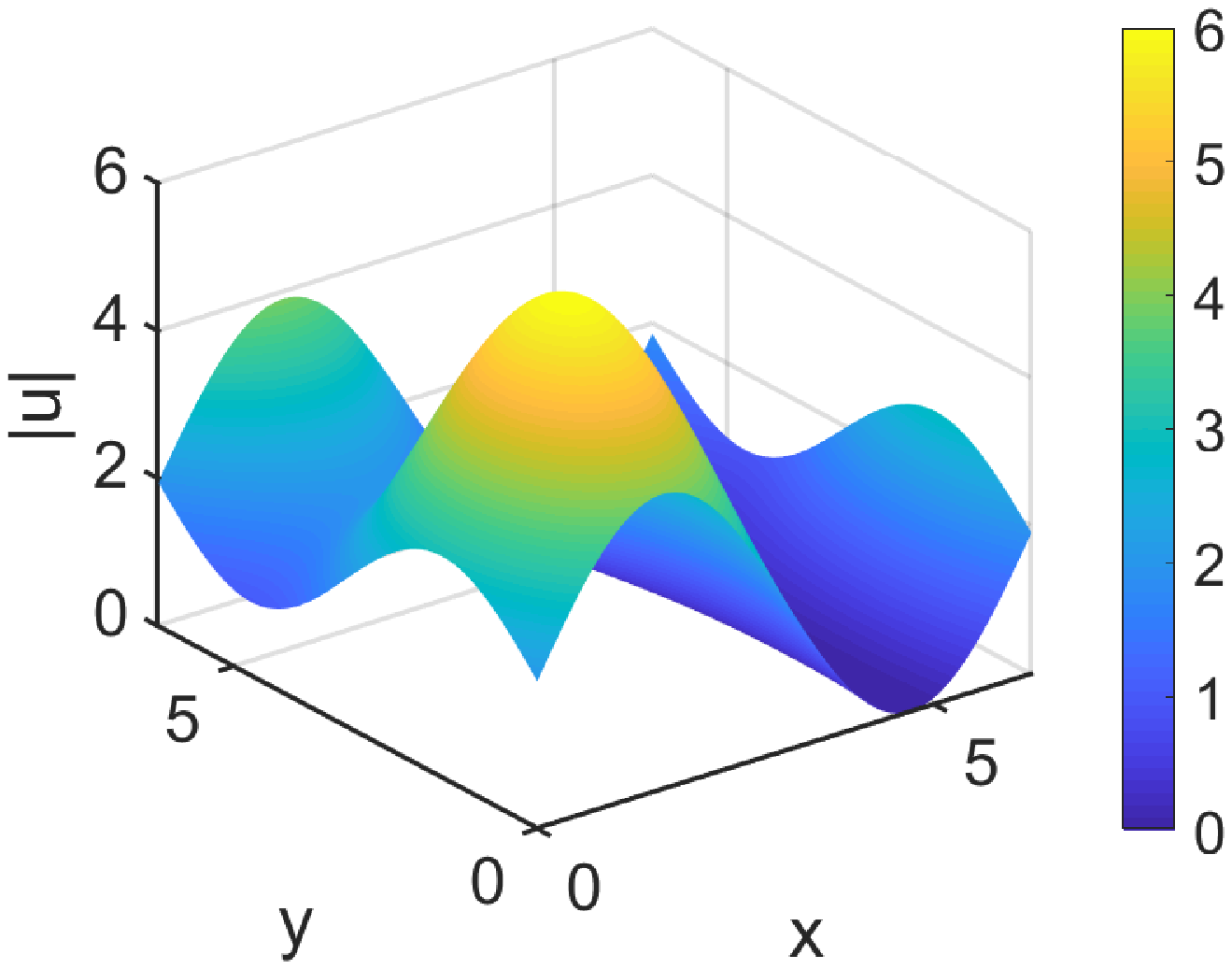}}
\subfloat[the coutour at $t = 0$]{\includegraphics[width=.48\textwidth]{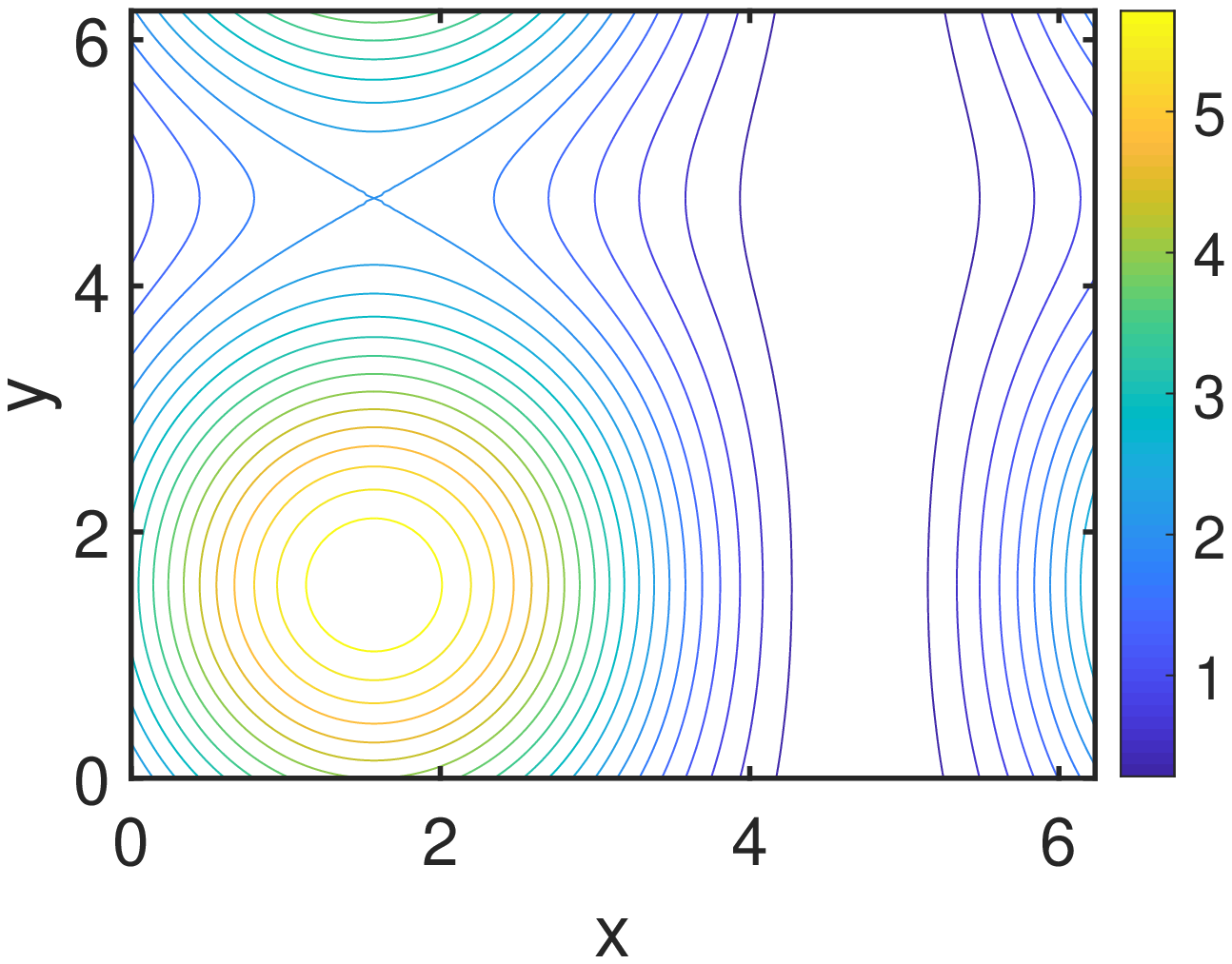}}
\caption{Profile and contour of singular solution at $t = 0$ with initial condition Eq. \eqref{exp:2dim-init}.}
  \label{fig:2d-bgn}
\end{figure}

\begin{figure}[htbp]
  \centering
\subfloat[the profile at $t = 0.108$]{\includegraphics[width=.48\textwidth]{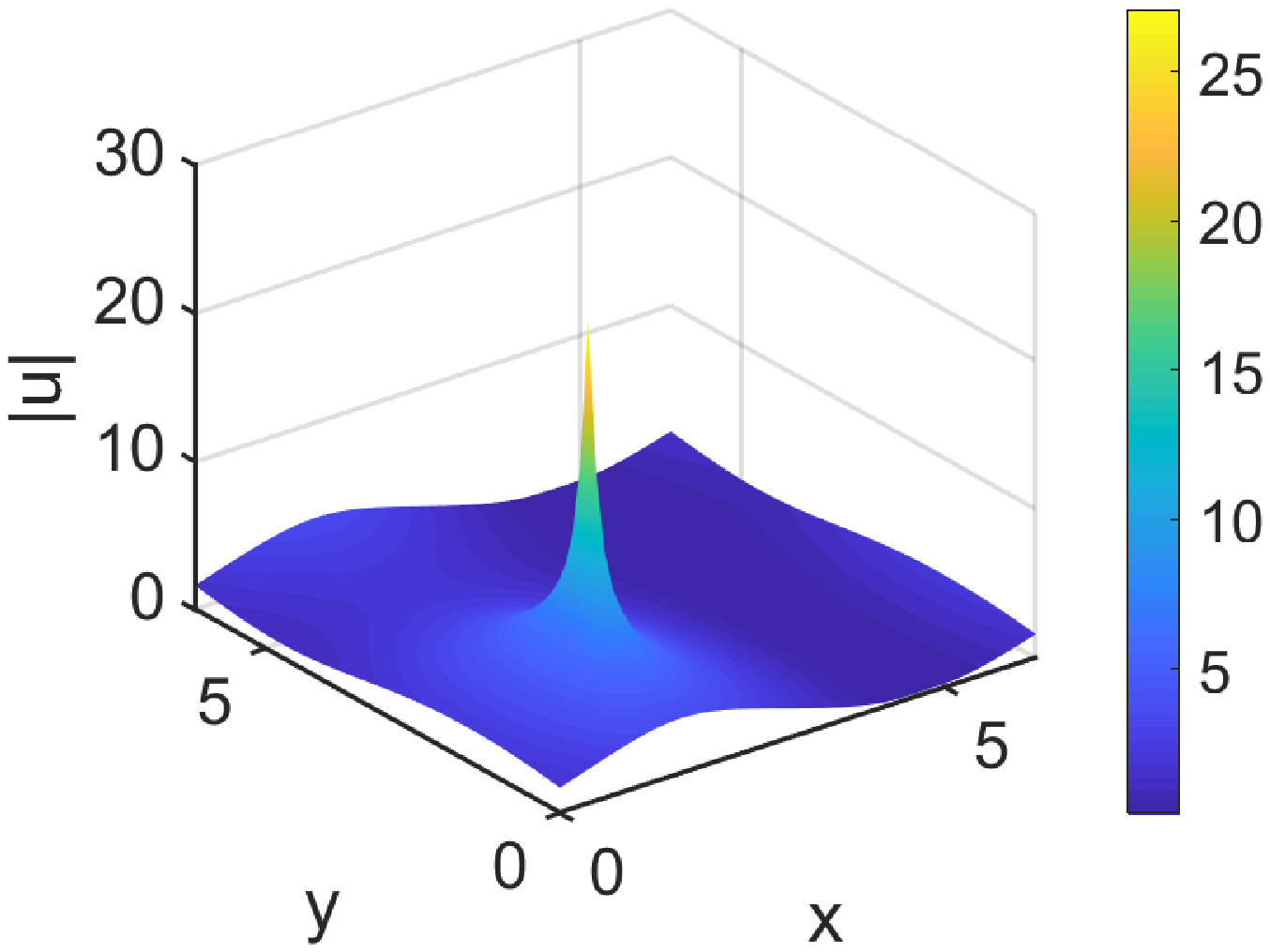}}
\subfloat[the contour at $t = 0.108$]{\includegraphics[width=.48\textwidth]{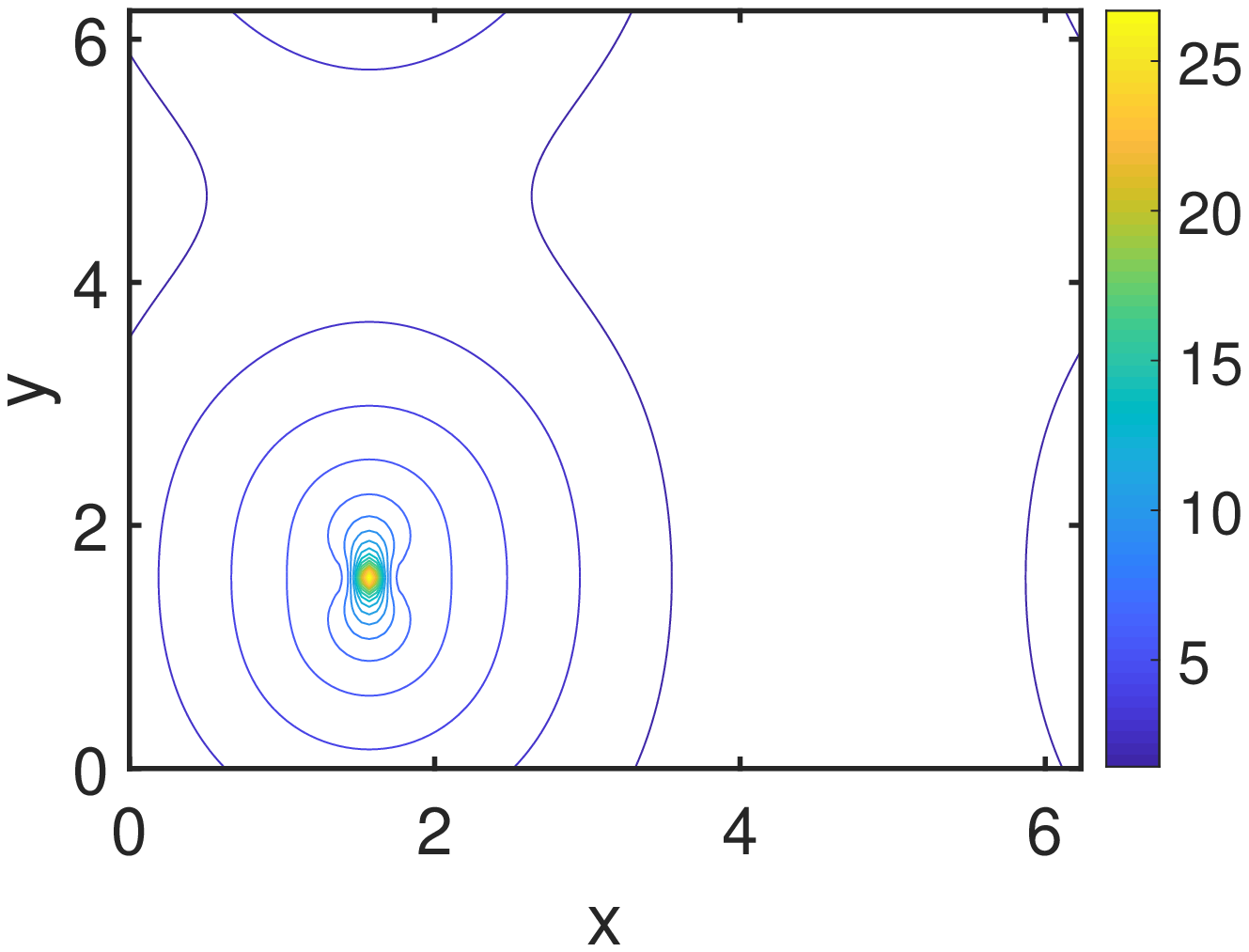}}
\caption{Profile and contour of singular solution at $t = 0.108$ with initial condition Eq. \eqref{exp:2dim-init}}
  \label{fig:2d-end}
\end{figure}

\begin{figure}[htbp]
  \centering
\subfloat[mass error]{\includegraphics[width=.48\textwidth]{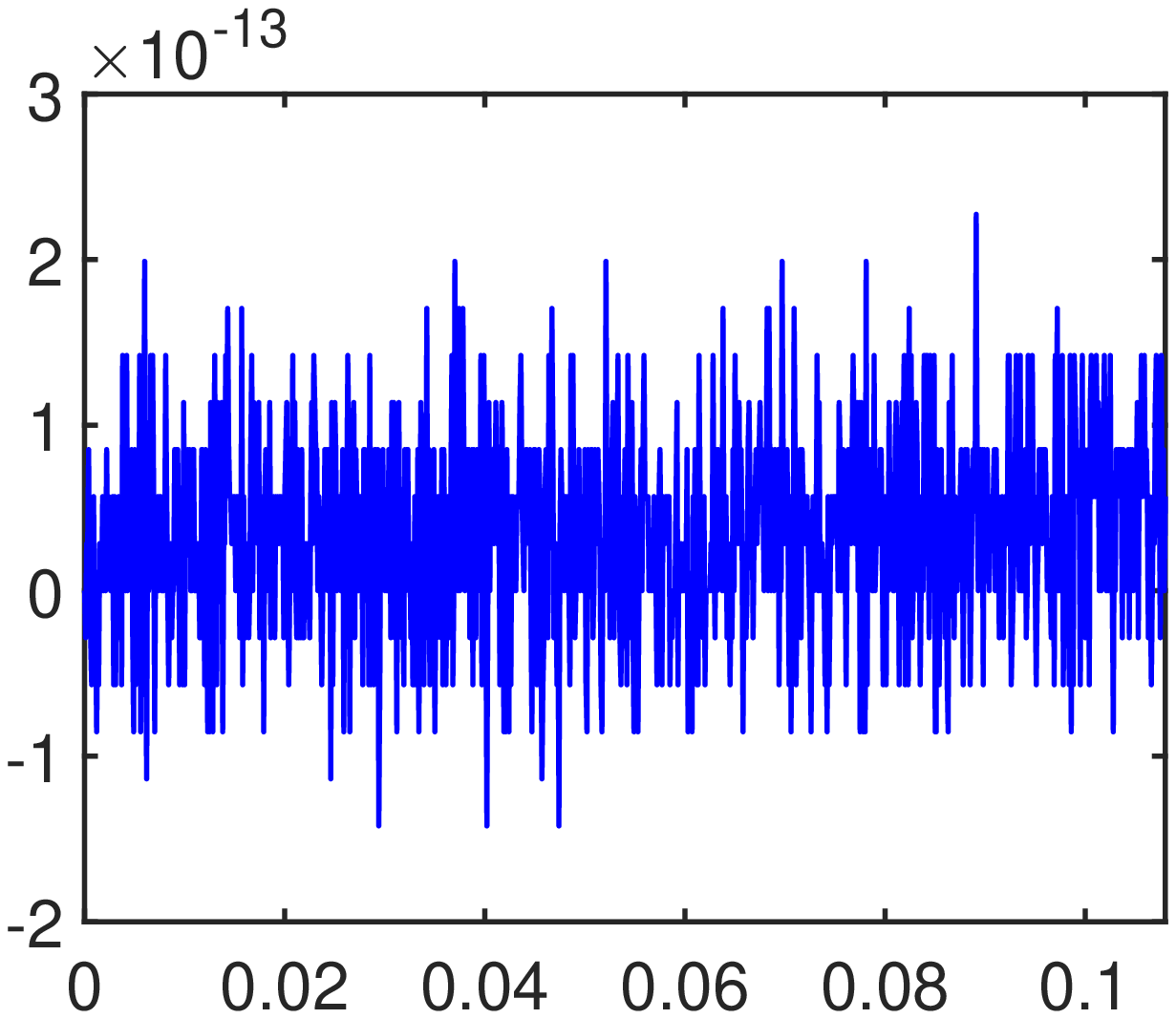}}
\subfloat[energy error]{\includegraphics[width=.48\textwidth]{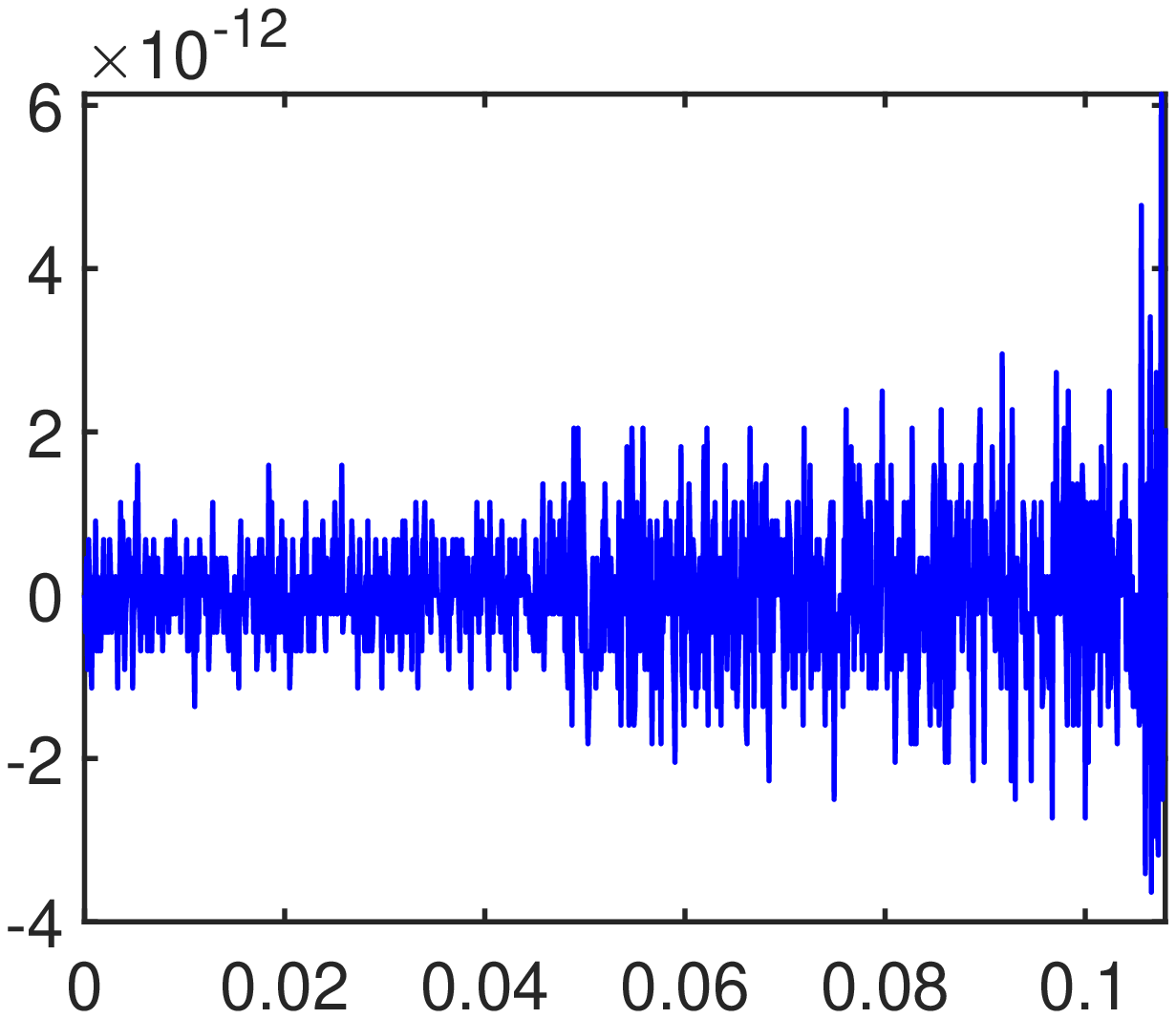}}
\caption{Error in mass and energy at with initial condition Eq. \eqref{exp:2dim-init}, ($|E(u_0)| \approx 1166.00$).}
  \label{fig:2d-error}
\end{figure}

\section{Conclusion}
In this paper, we studied a series of high order diagonally implicit Runge-Kutta methods which are mass and energy conservative for nonlinear Schr\"odinger equation. By applying the invariant
energy quadratizaton technique, the 1D and 2D NLS equations are
transformed into an equivalent reformulation which is subsequently discreted by the Fourier pseudospectral method in the space direction and so-called IEQ--DIRK schemes in the time direction. We gave rigorous proofs for the
conservative properties of the semi-discrete system and show that $M=\mathbf{0}$ is the necessary and sufficient condition for IEQ--DIRK schemes to preserve the mass and energy. The numerical results comfirmed the
theoretical analysis on convergence orders, conservative properties and longtime simulation stability, and verify the ability to capture the singularity of proposed schemes.
\bibliography{DIRK}

\end{document}